\documentclass[a4paper,12pt]{article}
\usepackage{amssymb,amsmath,amsthm,latexsym}
\usepackage{amsfonts}
	\usepackage{amsfonts}
\usepackage{graphicx}
\usepackage[pdftex,bookmarks,colorlinks=false]{hyperref}
\usepackage{verbatim}
\usepackage{caption}
\usepackage{subcaption}

\newtheorem{theorem}{Theorem}[section]

\newtheorem{corollary}[theorem] {Corollary}
\newtheorem{definition}[theorem]{Definition}

\newtheorem{lemma} [theorem]{Lemma}

\newtheorem{problem}{Problem}[theorem]

\newtheorem{remark}[theorem]{Remark}

\title{A CHARACTERISATION OF WEAK INTEGER ADDITIVE SET-INDEXERS OF GRAPHS}

\author{{\bf N K Sudev \footnote{Department of Mathematics, Vidya Academy of Science \& Technology, Thalakkottukara, Thrissur - 680501, email: {\em sudevnk@gmail.com}}} and {\bf K A Germina\footnote{Department of Mathematics, School of Mathematical \& Physical Sciences, Central University of Kerala, Kasaragod, email:{\em srgerminaka@gmail.com}}}}

\date{}

\begin{document}

\maketitle

\begin{abstract}
Acharya introduced the notion of set-valuations of graphs as a set analogue of the number valuations  of graphs. Also, we have the notion of set-indexers, integer additive set-indexers, $k$-uniform integer additive set-indexers, weak and strong integer additive set-indexers, weakly and strongly uniform integer additive set-indexers, completely uniform integer additive set-indexers and arbitrarily uniform integer additive set-indexers based on the cardinality of the labeling sets. In this paper, we study further characteristics of the graphs which admit weak integer additive set-indexers as extensions to our earlier studies and provide some useful results on these types of set-indexers.
\end{abstract}



\section{Introduction}\label{Introduction}
For all  terms and definitions, not defined specifically in this paper, we refer to \cite{FH}. Unless mentioned otherwise, all graphs considered here are simple, finite and have no isolated vertices.

The research in graph labeling problems commenced with the introduction of {\em $\beta$-valuations} in \cite{AR}. As a set analogue of number valuations of graphs, set-valuations are introduced in \cite{A1}, as follows. Let $G$ be a $(p,q)$-graph and let $X$, $Y$,$Z$ be non-empty sets. Then the functions $f:V(G)\to 2^X$, $f:E(G)\to 2^Y$ and $f:V(G)\cup E(G)\to 2^Z$ are called the {\em  set-assignments} of vertices, edges and elements of $G$ respectively. By a set-assignment of a graph, we mean any one of them.  A set-assignment is called a {\em set-labeling} or {\em set-valuation} if it is injective. A graph with a set-labeling $f$ is denoted by $(G,f)$  and is called a {\em set-labeled graph}.

For a $(p,q)$- graph $G=(V,E)$ and a non-empty set $X$ of cardinality $n$, a {\em set-indexer} of $G$ is defined in \cite{A1} as an injective set-valued function $f:V(G) \rightarrow2^{X}$ such that the function $f^{\oplus}:E(G)\rightarrow2^{X}-\{\emptyset\}$ defined by $f^{\oplus}(uv) = f(u ){\oplus}f(v)$ for every $uv{\in} E(G)$ is also injective, where $2^{X}$ is the set of all subsets of $X$ and $\oplus$ is the symmetric difference of sets.

\begin{theorem}
\cite{A1} Every graph has a set-indexer.
\end{theorem}

\begin{definition}\label{D1}{\rm
Let $\mathbb{N}_0$ denote the set of all non-negative integers. For all $A, B \subseteq \mathbb{N}_0$, the sum of these sets is denoted by  $A+B$ and is defined by $A + B = \{a+b: a \in A, b \in B\}$}.
\end{definition}

\begin{definition}\label{D2}{\rm
\cite{GA} An {\em integer additive set-indexer} (IASI, in short) is defined as an injective function $f:V(G)\rightarrow 2^{\mathbb{N}_0}$ such that the induced function $g_f:E(G) \rightarrow 2^{\mathbb{N}_0}$ defined by $g_f (uv) = f(u)+ f(v)$ is also injective}.
\end{definition}

\begin{definition}\label{D3}{\rm
\cite{GS1} The cardinality of the labeling set of an element (vertex or edge) of a graph $G$ is called the {\em set-indexing number} of that element.}
\end{definition}

\begin{definition}\label{DU}{\rm
\cite{GA} An IASI is said to be {\em $k$-uniform} if $|g_f(e)| = k$ for all $e\in E(G)$. That is, a connected graph $G$ is said to have a $k$-uniform IASI if all of its edges have the same set-indexing number $k$. In particular, we say that a graph $G$ has an {\em arbitrarily $k$-uniform IASI} if $G$ has a $k$-uniform IASI  for every positive integer $k$.}
\end{definition}

The characteristics of a special type of 1-uniform IASI graphs, called {\em sum square graphs}, has been studied in \cite{GA}. A characterisation of 2-uniform IASI graphs has been done in \cite{TMKA}. As a result, the following theorem was proved.

\begin{theorem}\label{T-2U}
\cite{TMKA} A graph $G$ admits a 2-uniform IASI if and only if it is bipartite.
\end{theorem}

In the context of IASIs, in \cite {GS1}, we established the following lemma and introduced the notions following it.

\begin{lemma}\label{L-Card}
\cite{GS1} Let $A$ and $B$ be two non-empty finite subsets of $\mathbb{N}_0$. Then $max(|A|,|B|)\le |A+B|\le |A|.|B|$. Therefore, for an integer additive set-indexer $f$ of a graph $G$, we have $max(|f(u)|, |f(v)|)\le |g_f(uv)|= |f(u)+f(v)| \le |f(u)| |f(v)|$, where $u,v\in V(G)$.
\end{lemma}

\begin{definition}{\rm
\cite{GS1} An IASI $f$ is said to be a {\em weak IASI} if $|g_f(uv)|=max(|f(u)|,|f(v)|)$ for all $u,v\in V(G)$ and is said to be a {\em strong IASI} if $|g_f(uv)|=|f(u)|.|f(v)|$ for all $u,v\in V(G)$.}
\end{definition}

\begin{definition}{\rm
\cite{GS1} A weak  IASI (or a strong  IASI) is said to be {\em weakly uniform IASI} (or {\em strongly uniform IASI}) if $|g_f(uv)|=k$, for all $u,v\in V(G)$ and for some positive integer $k$.  A graph which admits a weak (or strong) IASI may be called a {\em weak (or strong) IASI graph}.}
\end{definition}

\begin{theorem}\label{T-WU}
\cite{GS1} For $k>1$, graph $G$ admits a weakly $k$-uniform IASI if and only if it is bipartite. No non-bipartite graph $G$ has a weakly $k$ uniform IASI, for $k>1$ 
\end{theorem}

\begin{theorem}\label{T-AU}
\cite{GS1} A weakly $k$-uniform IASI graph is also an arbitrarily $k$-uniform IASI graph.
\end{theorem}

\begin{theorem}\label{T-WUAU}
\cite{GS1} A graph $G$ admits a $k$-uniform IASI if and only if $k$ is odd or $G$ is bipartite.
\end{theorem}


\section{New Results on Weak IASI Graphs}

In this paper, we establish more results on weak IASI graphs and graph operations. We proceed by defining the following notions.

\begin{definition}{\rm
An element (a vertex or an edge) of graph $G$ which has the set-indexing number 1 is called a {\em mono-indexed element} of that graph.}
\end{definition}

\begin{definition}{\rm
The {\em sparing number} of a graph $G$ is defined to be the minimum number of mono-indexed edges required for $G$ to admit a weak IASI and is denoted by $\varphi(G)$.}
\end{definition}

\begin{theorem}\label{T-WSG}
If a graph $G$ is a weak IASI graph, then any subgraph $H$ of $G$ is also a weak IASI graph.  
\end{theorem}
\begin{proof}{\rm
We have already proved the corresponding result for weakly $k$-uniform IASI graphs in \cite{GS1}. Let $G$ be a graph which admits a weak IASI and $H$ be a subgraph of $G$. Let $f^\ast$ be the restriction of $f$ to $V(H)$. Then $g_{f^\ast}$ is the corresponding restriction of $g_f$ to $E(H)$. Then clearly, $f^\ast$ is an induced set-indexer on $H$, which is a weak IASI on $H$. Hence, $H$ is a weak IASI graph.}
\end{proof}

\begin{remark}\label{R-WSG}{\rm
As the contrapositive statement of Theorem \ref{T-WSG}, we notice that if $G$ is a graph which has no weak IASI, then any supergraph of $G$ does not have a weak IASI.}
\end{remark}

Now, we proceed to verify the existence of weak IASI for different classes of graphs, as all graphs, in general, do not admit a weak IASI.

\begin{theorem}\label{T-WBP}
All bipartite graphs admit a weak IASI.
\end{theorem}
\begin{proof}
Let $G$ be a bipartite graph with bipartition $(V_1,V_2)$ of the vertex set $V(G)$. Assign distinct singleton sets to the vertices in $V_1$ and distinct non-singleton sets to the vertices in $V_2$. This IASI is clearly a weak IASI for $G$.
\end{proof} 

\begin{remark}\label{R-BP}{\rm
Due to Theorem \ref{T-WBP}, all paths, trees and even cycles admit a weak IASI. We also, observe that the sparing number of bipartite graphs is $0$.}
\end{remark}

\begin{theorem}\label{T-SB1}
If a connected graph $G$ admits a weak IASI, then $G$ is bipartite or $G$ has at least one mono-indexed edge. 
\end{theorem}
\begin{proof}
Assume that the graph $G$ is a weak IASI graph. Then, by the definition of weak IASI, each edge of $G$ has an end vertex which is mono-indexed. Let $V_1$ be the set of all mono-indexed vertices of $G$ and $V_2=V(G)-V_1$. If $(V_1,V_2)$ forms a bipartition of $V(G)$, then the proof is complete. Hence, assume that $(V_1,V_2)$ is not a bipartition of $V(G)$. Since every vertex of $V_2$ has the set-indexing number greater than 1, every edge connecting two vertices in $V_2$, if exists, has both the end vertices with set-indexing number greater than 1, which is a contradiction to the hypothesis that $G$ has a weak IASI. Hence, no two vertices of $V_2$ can be adjacent. Therefore, at least two vertices in $V_1$ must be adjacent. If $e$ is such an edge in $G$, with both end vertices in $V_1$, then the set-indexing number of $e$ is 1. That is, $G$ has at least one mono-indexed edge.
\end{proof}

\begin{theorem}
Let $G$ be a bipartite graph which admits a weak IASI and let $u,v$ be two non-adjacent vertices in $G$. Then, $G+uv$ is a weak IASI graph if and only if $G+uv$ is bipartite or $uv$ is a mono-indexed edge.
\end{theorem}
\begin{proof}\label{T-WKBP}
Let $G$ be a bipartite graph which admits a weak IASI and let $u,v$ be two non-adjacent vertices in $G$. Assume that $G_1=G+uv$ be a weak IASI graph. Then by Theorem \ref{T-SB1}, $G_1$ is bipartite or the edge $uv$ lies in the same partition of mono-indexed vertices in $G$. That is, if $G_1$ is non-bipartite, then $uv$ is mono-indexed.

Conversely, if $G_1=G+uv$ is bipartite, then by Theorem \ref{T-WBP}, it has a weak IASI. If not, then by hypothesis, $uv$ is mono-indexed. Since $G$ is a weak IASI graph and $uv$ is mono-indexed, $G+uv$ is a weak IASI graph. Hence, if $G+uv$ is bipartite or $uv$ is mono-indexed, then $G+uv$ is a weak IASI graph.
\end{proof}

\begin{corollary}\label{C-WKBP}
Let a complete bipartite graph $K_{m,n}$ be a weak IASI graph. Then, a graph $G=K_{m,n}+e$ admits a weak IASI if and only if $e$ is a mono-indexed edge in $G$.
\end{corollary}
\begin{proof}
Since $K_{m,n}$ is complete bipartite, $G$ can not be a bipartite graph. Then, by Corollary \ref{T-WKBP}, $e$ must be a mono-indexed edge of $G$. 
\end{proof}
Figure \ref{W-NB} depicts a weak IASI for a non-bipartite graph.

\begin{figure}[h!]
   \centering
    \includegraphics[width=0.5\textwidth]{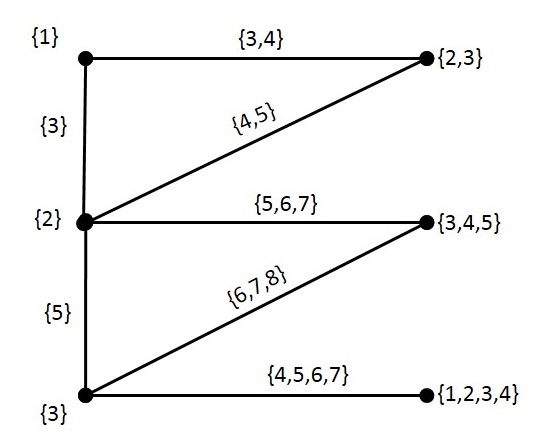}
     \caption{A non-bipartite graph which has a weak IASI.\label{W-NB}}
\end{figure}

\begin{theorem}\label{T-WKN}
The complete graph $K_n$ admits a weak IASI if and only if the maximum number of mono-indexed edges of $K_n$ is $\frac{1}{2}(n-1)(n-2)$.
\end{theorem}
\begin{proof}
If $K_n$ is 1-uniform, then the result is trivial. Hence, we consider the IASI graphs that are not 1-uniform. 

Assume that $K_n$ admits a weak IASI. Then, every edge of $K_n$ must have at least one end vertex of set-indexing number 1. Without loss of generality, label the vertex $v_1$ by a singleton set. Since $v_2$ is adjacent to $v_1$, label $v_2$ by a non-singleton set. Since all other vertices $v_3, v_4, \cdots, v_n$ are adjacent to $v_2$ in $K_n$ and $K_n$ admits a weak IASI, none of these vertices can be labeled by a non-singleton set. Therefore, only the edges incident on $v_2$ can have a set-indexing number greater than 1. That is, only $(n-1)$ edges in $K_n$ have a set indexing number greater than 1. Hence, the number of mono-indexed edges of $K_n$ is $\frac{1}{2}n(n-1)-(n-1)=\frac{1}{2}(n-1)(n-2)$. 

Conversely, assume that the number mono-indexed edges of $G$ is $\frac{1}{2}(n-1)(n-2)$. Therefore, the number edges of $G$ having set-indexing numbers greater than 1 is $\frac{1}{2}n(n-1)-\frac{1}{2}(n-1)(n-2)= (n-1)$. Since $K_n$ is an $(n-1)$-regular graph, all these edges must incident on the same vertex, say $v_i$. That is, the vertex $v_i$ is the only vertex that has a set-indexing number greater than 1. This set-indexer is a weak IASI for $K_n$.
\end{proof}

\begin{remark}{\rm
From Theorem \ref{T-WKN}, we observe that the sparing number of a complete graph $K_n$ is $\frac{1}{2}(n-1)^2$.}
\end{remark}

The Figure \ref{W-CG}, illustrates the admissibility of a weak IASI for complete graphs.

\begin{figure}[h!]
   \centering
    \includegraphics[width=0.5\textwidth]{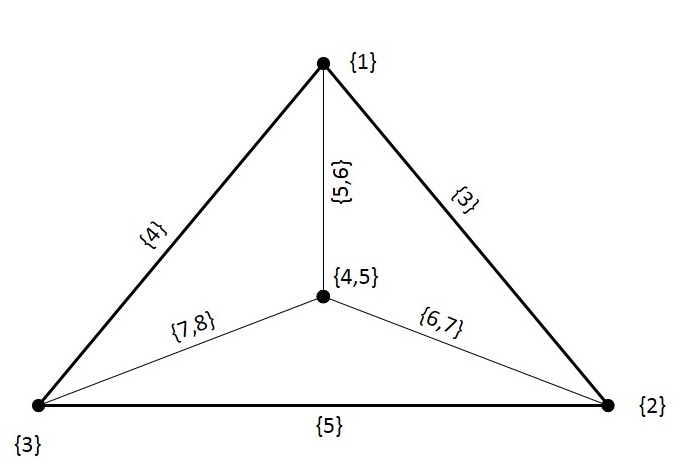}
     \caption{Existence of a weak IASI for complete graphs\label{W-CG}}
\end{figure}

Since bipartite graphs support weak IASI, we now concentrate on non-bipartite graphs. First, we establish a necessary and sufficient condition for an odd cycle to have a weak IASI.

\begin{theorem}\label{T-WUOC}
An odd cycle $C_n$ has a weak IASI if and only if it has at least one mono-indexed edge. 
\end{theorem}
\begin{proof}
Let $C_n:v_0v_1v_2\cdots v_n$ be an odd cycle which admits a weak IASI. Then, $n=2m+1$ for some non-negative integer $m$. Without loss of generality, assign a singleton set to the vertex $v_1$. Since $C_n$ admits a weak IASI, the vertex $v_2$ adjacent to $v_1$ must be assigned by a non-singleton set, where $k>1$. Let $v_3$ be the vertex which is adjacent to $v_2$ other than $v_1$. Then, by definition of weak IASI, $v_3$ must be assigned by a singleton set other than the one assigned to $v_1$. If $v_4$ is the vertex adjacent to $v_3$, the vertex $v_4$ must be labeled by a non-singleton set, which has not already been used. Proceeding like this, we see that the vertices $v_1, v_3,v_5,\cdots, v_{2m-1}$ have the set-indexing number 1 and $v_2,v_4,v_6,\cdots, v_{2m}$ have set-indexing numbers greater than 1. Now, since $C_n$ admits a weak IASI, the vertex $v_{2m+1}$ should necessarily be labeled by a singleton set, which has not already been used. Therefore, the edge $v_{2m+1}v_1$ is mono-indexed.

Conversely, assume that $C_n:v_1v_2,v_3,\cdots, v_{n-1}v_nv_1$ be an odd cycle which can have a mono-indexed edge. Let $f: V(G)\to 2^{\mathbb{N}_0}$ be a labeling of vertices to distinct subsets of  $\mathbb{N}_0$ defined as follows. Without loss of generality, label the vertices $v_1$ and $v_n$ by distinct singleton sets. Hence, the edge $v_1v_n$ is mono-indexed. Then, label the vertices $v_2$ and $v_{n-1}$ by distinct non-singleton sets; label the vertices $v_3$ and $v_{n-2}$ by distinct singleton sets that have not been already used;  and label the vertices $v_4$ and $v_{n-3}$  by distinct non-singleton sets, not already used before. Since $C_n$ is an odd cycle, in the succeeding steps, the vertices that are adjacent to the preceding labeled vertices by distinct singleton sets and distinct non-singleton sets, (not already used before), we reach a single vertex $v_r$, where $r={\frac{1}{2}(n-1)}$ and the vertices adjacent to whose end vertices are labeled by singleton sets (or non-singleton sets). Then, label the vertex $v_r$ by a non-singleton set (or a singleton set in the latter case), that have not been already used. This set-indexer $f$ is a weak IASI of $C_n$.
\end{proof}

Invoking Remark \ref{R-BP} and Theorem \ref{T-WUOC}, we have

\begin{remark}
The sparing number of a a cycle $C_n$ is given by 
\[ f(v)= \left\{
\begin{array}{l l}
 	1 & \quad \text{if $C_n$ is odd}\\
    0 & \quad \text{if $C_n$ is even}
 \end{array} \right.\].
\end{remark}

An interesting question here is to know the number of mono-indexed edges in cycles. Following theorem answers this question.

\begin{theorem}\label{T-NME}
Let $C_n$ be a cycle of length $n$ which admits a weak IASI, for a positive integer $n$. Then, $C_n$ has an odd number of mono-indexed edges when it is an odd cycle and has even number of mono-indexed edges, when it is an even cycle. 
\end{theorem}  
\begin{proof}
Let $C_n:v_1v_2v_3\cdots v_nv_1$ be a cycle of length $n$ which has a weak IASI. 

\noindent {\em Case 1:} If $n$ is even, then it can be considered as a bipartite graph with bipartition $(X_1,X_2)$ in such a way that every vertex of $X_1$ is mono-indexed and every vertex of $X_2$ has the set-indexing number greater than 1. If we replace the labeling set of a vertex $v$ in $X_2$ by a singleton set, then the set-indexing number of two edges incident on $v$ becomes 1. Whenever we replace the non-singleton labeling set of a vertex by a singleton set, in $X_2$, two edges attains the set-indexing number 1. Hence, $C_n$ can only have even number of mono-indexed edges.

\noindent {\em Case 2:} If $n$ is odd, then by Theorem \ref{T-WUOC}, it has atleast one mono-indexed edge. If we replace a non-singleton labeling set of a vertex $v$ in $C_n$ by a singleton set, then two edges incident on $v$ attains a set-indexing number 1.  Whenever we replace a non-singleton labeling set of any vertex of $C_n$ by a singleton set, the number of edges having set-indexing number is increased by 2. Therefore, $C_n$ has odd number of mono-indexed edges.
\end{proof}

\begin{definition}{\rm
\cite{KDJ} Let $G$ be a connected graph and let $v$ be a vertex of $G$ with $d(v)=2$. Then, $v$ is adjacent to two vertices $u$ and $w$ in $G$. If $u$ and $v$ are non-adjacent vertices in $G$, then delete $v$ from $G$ and add the edge $uw$ to $G-\{v\}$. This operation is known as an {\em elementary topological reduction} on $G$.}
\end{definition}

\begin{theorem}
Let $C_n$ be a cycle which admits a weak IASI. Then, any graph $H$, obtained by applying finite number of elementary topological reduction on $C_n$, admits a weak IASI. 
\end{theorem}
\begin{proof}
We prove the theorem in three steps.

\noindent {\em Step-1:} If $C_n$ has no mono-indexed edge, then proceed to Step 2. If $C_n$ consists of mono-indexed edges, then delete one end vertex $v$ of a mono-indexed edge. Then, two edges will be deleted and $C_n-{v}$ has two pendant vertices, of which, one vertex has the set-indexing number 1 and the other vertex has a set-indexing number 1 or greater than 1. Join these vertices by an edge. The resultant graph is a cycle, which also has a weak IASI. Repeat this process until all edges having set-indexing number 1 have been deleted.

\noindent {\em Step-2:} At this stage, we have a cycle $C_j,j<i$ which contains no mono-indexed edges. The cycle $C_j$ also admits a weak IASI. Now, delete a vertex $v$ that is not mono-indexed and the two edges incident on it. Hence, $C_j-v$ has two pendant vertices each of which has the set-indexing number 1. Repeat this steps until all possible such vertices are deleted.

\noindent {\em Step-3:} Repeat Step-1 and Step-2 successively until we get a cycle $C_3$ of length 3, all of whose vertices are mono-indexed (or two vertices are mono-indexed and one vertex has set-indexing number greater than 1). Hence $C_3$ also admits a weak IASI. 
\end{proof}

Figure \ref{W-CHom} illustrates the existence of weak IASI for the graphs $H_1$ and $H_2$ obtained by applying a finite number of elementary toplogical reduction on the graph $G$ which is cycle of length 5.

\begin{figure}[h!]
   \centering
    \includegraphics[width=1\textwidth]{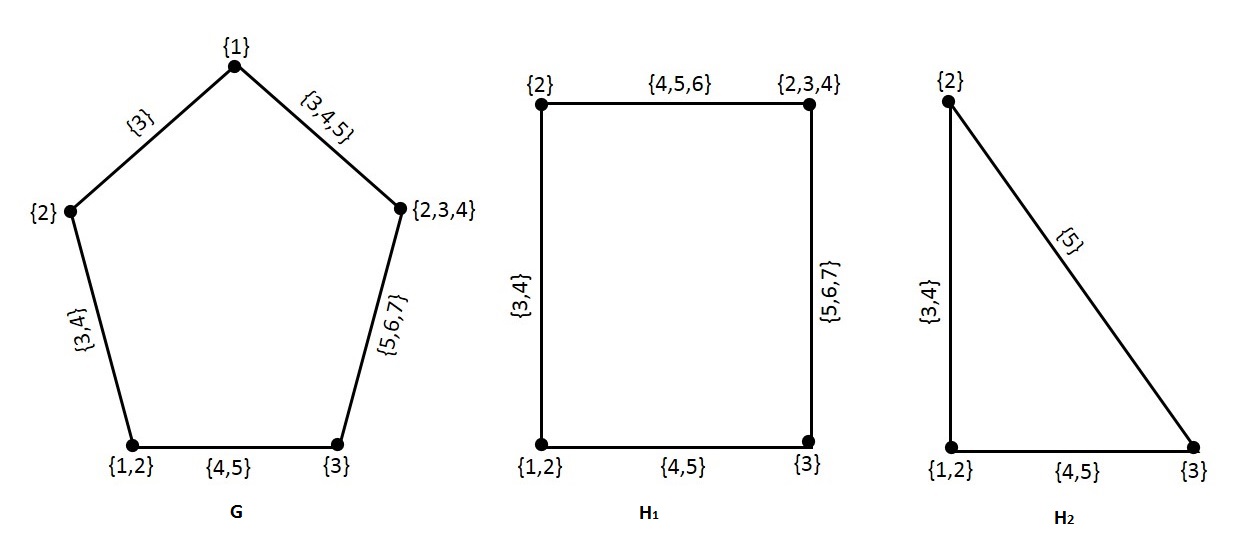}
     \caption{Existence of a weak IASI for the graphs obtained after finite number of topological reductions on $G$.\label{W-CHom}}
\end{figure}

By {\em edge contraction operation} in $G$, we mean an edge, say $e$, is removed and its two incident vertices, $u$ and $v$, are merged into a new vertex $w$, where the edges incident to $w$ each correspond to an edge incident to either $u$ or $v$. We establish the following theorem for the graphs obtained by contracting the edges of a given graph $G$.

\begin{theorem}
Let $G$ be a weak IASI graph and let $G'$ be the graph obtained by contracting an edge $e$ of $G$. $G'$ admits a weak IASI if and only if $G'$ is bipartite or $e$ is a mono-indexed edge.
\end{theorem}
\begin{proof}
Let $G$ be a weak IASI graph and let $G'$ be a graph obtained from $G$ by deleting an edge $e$ of $G$ and identifying the end vertices of $e$. Label the new vertex by the set-label of the deleted edge. If $e$ is a mono-indexed edge, then deleting $e$ and identifying its end vertices will result in a new vertex $w$, labeled with set-labeling of $e$, which is mono-indexed. Then $G'$ is a weak IASI graph. 

Conversely, assume that $G'$ is a weak IASI graph.  If $G'$ is bipartite, then proof is complete. Hence, assume that $G'$ is not bipartite. Let $e=uv$ an edge in $G$ such that $G'=G\circ e$. Now, remove $e$ form $G$ and identify the vertices $u$ and $v$ to a new vertex $w$. Label the vertex $w$ by the labeling set of $uv$. We have to show that both $u$ and $v$ are mono-indexed. Assume the contrary. Since $G$ admits a weak IASI, $u$ and $v$ simultaneously can not have a set-indexing number greater than 1. Hence, without loss of generality, let $u$ is mono-indexed and $v$ is not mono-indexed. Since $G\circ e$ is non-bipartite, $u$ must be adjacent to at least one vertex, say $x$, which is not mono-indexed. Therefore, in $G\circ e$ the edge $xw$ has both end vertices which are not mono-indexed, which is a contradiction. Therefore, both $u$ and $v$ must be mono-indexed. Hence, $e$ is mono-indexed.
\end{proof}
Figure \ref{W-GEC-e} illustrates the admissibility of weak IASI by a graph obtained from contracting a mono-indexed edge in a weak IASI graph $G$. 

\begin{figure}[h!]
   \centering
    \includegraphics[width=1\textwidth]{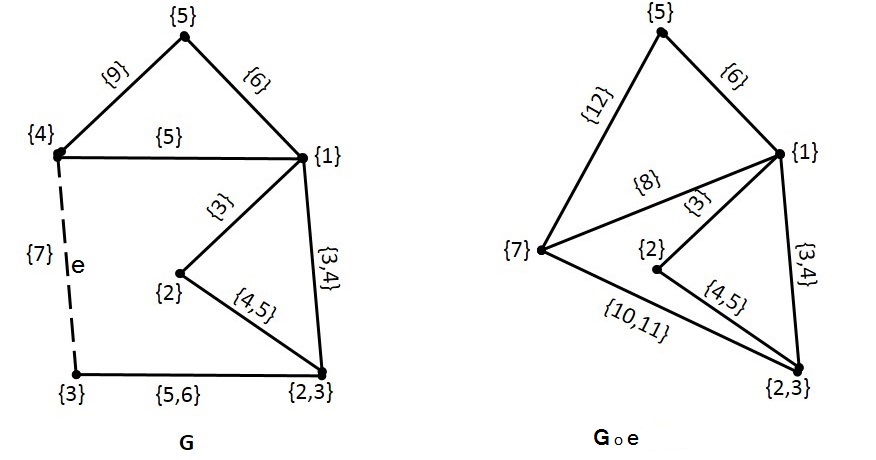}
     \caption{Existence of a weak IASI for a graph obtained contracting an edge of a weak IASI graph $G$.\label{W-GEC-e}}
\end{figure}


\section{Conclusion}\label{Sec_Conclusion}

In this paper, we have given some characteristics of the graphs which admit weak IASIs. We have identified some problems related to weak IASI graphs which demands further investigation. They are 
\begin{problem}
Discuss the existence of weak IASI for different graph operations and to determine their sparing number.
\end{problem}

\begin{problem}
Discuss the existence of weak IASI for Cartesian product of weak IASI graphs and to determine their sparing number.
\end{problem} 

\begin{problem}
Determine the necessary and sufficient conditions for the given graphs to admit weak IASIs.
\end{problem}

\begin{problem}
Determine the sparing number of a general connected (non-bipartite) graph.
\end{problem}

One may establish more properties and characteristics of weak and strong IASIs, both uniform and non-uniform. More studies may be done in the field of IASI when the ground set $X$ is finite instead of $\mathbb{N}_0$. We have formulated some conditions for some particular graph classes to admit weak and strong IASIs. The problems of establishing the necessary and sufficient conditions for various graphs and graph classes to have certain IASIs still remain unsettled. All these facts highlight a great scope for further studies in this area. 


\end{document}